\documentclass[11pt]{article}
\usepackage{amsthm,amsfonts,amsmath,amssymb}
\usepackage{mdwlist}
\usepackage{mathrsfs,tikz}
\usepackage{dsfont}
\usepackage[
    colorlinks=true,
    linkcolor=blue,
    citecolor=blue,
    urlcolor=blue
]{hyperref}
\newtheorem{thm}{Theorem}[section]

\newtheorem{lem}[thm]{Lemma}
\newtheorem{cor}[thm]{Corollary}

\newtheorem{rem}[thm]{Remark}

\numberwithin{equation}{section}

\begin{document}
\title{The Quadratic
Conjecture and $p$-groups of maximal class\thanks{This work was supported by NSFC (No. 12371022 \& 12271318)}}
\author{Jingjing Duan and Lijian An\thanks{Corresponding author. e-mail: anlj@sxnu.edu.cn }\\
Department of Mathematics,
Shanxi Normal University\\
Taiyuan, Shanxi 030031, P. R. China\\
 }

\maketitle

\begin{abstract}
The Quadratic Conjecture  for finite \(p\)-groups, which is equivalent to the Strong Oliver Conjecture, asserts that if a finite \(p\)-group \(G\) acts faithfully on an \(F\)-module \(V\), then the center \(Z(G)\) contains a quadratic element. In this paper, we investigate the structure of maximal class \(p\)-groups and establish that they satisfy the Quadratic Conjecture.

\medskip
 \noindent{\it 2000
Mathematics subject classification:} 20D15.
\end{abstract}

\baselineskip=16pt

\section{Introduction}

The structure of finite $p$-groups and their actions on modules
in characteristic $p$ are fundamental topics in finite group
theory. Finite $p$-groups are also closely connected with fusion
systems and the theory of classifying spaces in algebraic
topology. In their study of $p$-completed classifying spaces,
Broto, Levi, and Oliver developed the theory of $p$-local finite
groups~\cite{BLO2003}. In this framework, saturated fusion
systems encode fusion among subgroups of a finite $p$-group,
while centric linking systems provide the additional categorical
structure needed to construct the associated classifying spaces.
A fundamental question was whether every saturated fusion
system admits an associated centric linking system, unique up
to isomorphism.

In 2004, Oliver introduced a characteristic subgroup
$\mathfrak{X}(S)$ of a finite $p$-group $S$ and formulated
what is now known as Oliver's $p$-group
conjecture~\cite{Oliver2004}. For odd primes, this conjecture
provides a purely group-theoretic sufficient condition for the
vanishing of certain higher derived limits arising in the
obstruction theory for centric linking systems. Oliver showed
that the conjecture would imply the relevant vanishing results
and hence the existence and uniqueness of these linking systems.
Subsequently, Chermak established the existence and uniqueness
of centric linking systems associated with arbitrary saturated
fusion systems, and Oliver presented a proof of this theorem
using obstruction theory~\cite{Chermak2013,Oliver2013}.
Although the original topological problem has thus been
resolved, the underlying $p$-group conjecture remains of
independent interest.

A significant step toward a representation-theoretic approach
was taken by Green, H\'ethelyi, and Lilienthal, who gave a
modular reformulation of Oliver's $p$-group
conjecture~\cite{GHL2008}. Building on this approach, Green,
H\'ethelyi, and Mazza proposed the strong Oliver conjecture
and introduced its module-theoretic counterpart, known as the
\emph{Quadratic Conjecture}~\cite{GHM2011}.

To state the latter, let $G$ be a finite $p$-group and let
$V$ be a faithful $\mathbb{F}_pG$-module. The module $V$ is
called an $F$-module if there exists a nontrivial elementary
abelian subgroup $E\leq G$ such that
\begin{equation}\label{eq:offender}
    |E|\,|C_V(E)|\geq |V|.
\end{equation}
Such a subgroup $E$ is called an \emph{offender}.
A nonidentity element $g\in G$ is \emph{quadratic} on $V$
if $(g-1)^2V=0$, equivalently, if its minimal polynomial
on $V$ is $(X-1)^2$.

The Quadratic Conjecture asserts that whenever $V$ is a
faithful $F$-module for $G$, there exists a quadratic element
\begin{equation}\label{eq:quadratic-conjecture}
    1\neq g\in\Omega_1(Z(G)).
\end{equation}
This conjecture is automatic for $p=2$; for $p=3$,
it agrees with the earlier module-theoretic
reformulation~\cite{GHL2008,GHM2011}.

Several important cases of the Quadratic Conjecture have
been established. Although their 2010 paper was phrased
in terms of the earlier, weaker modular conjecture,
the argument of Green, H\'ethelyi, and Mazza in fact
proves the Quadratic Conjecture for metabelian $p$-groups
and for $p$-groups of nilpotency class at most
four~\cite{GHM2010,GHM2011}.

In their subsequent work, the same authors proved the
Quadratic Conjecture for groups generated by their
abelian normal subgroups and, more generally, for
groups satisfying
\[
    \Omega_1\bigl(Z(\mathcal{Y}(G))\bigr)
    =\Omega_1(Z(G)),
\]
where $\mathcal{Y}(G)$ is the characteristic subgroup
defined in~\cite{GHM2011}.

They also established the conjecture when
$G'\cap\Omega_1(G)$ is abelian, and in particular
for powerful $p$-groups. Furthermore, they verified
it for the Sylow $p$-subgroups of all symmetric
groups and of all general linear groups
$\mathrm{GL}_n(\mathbb{F}_q)$, including both
defining-characteristic and cross-characteristic
cases~\cite{GHM2011}. These results also encompass
certain iterated wreath products of cyclic $p$-groups
and suitable direct products of such groups.

Further advances have clarified the structure and
limitations of possible proofs. For $p\geq5$, Lynd
showed that every faithful $F$-module admits a
quadratic offender $E$ that is normal in its normal
closure, namely,
\[
    E\trianglelefteq\langle E^G\rangle
    \trianglelefteq G
\]
\cite{Lynd2013}.
This is a structural result about quadratic offenders
and does not, by itself, prove the existence of a
central quadratic element.

In a different direction, Green and Lynd constructed
a group $G$ of order $3^{49}$ for which the proposed
Weakly Closed Conjecture fails for a faithful module,
while the Quadratic Conjecture holds for every
faithful $\mathbb{F}_3G$-module~\cite{GreenLynd2013}.
Thus, the Weakly Closed Conjecture is strictly
stronger than the Quadratic Conjecture and cannot
serve as a universal route to its proof.

In our previous work~\cite{DuanAn2025}, we established the
Quadratic Conjecture for certain $p$-groups of maximal
class. More precisely, we proved that the conjecture
holds for every $p$-group of maximal class and order
$p^n$ whenever
\[
    n\leq 8
    \qquad\text{or}\qquad
    n\geq\max\{2p-6,\,p+2\}.
\]
In particular, this settled the Quadratic Conjecture
for all $p$-groups of maximal class when $p\leq 7$.
In the present paper, we remove these restrictions
and establish the Quadratic Conjecture for all
$p$-groups of maximal class.
Our main result is the following.

\begin{thm}\label{thm:maximal_class}
Every finite $p$-group of maximal class satisfies
the Quadratic Conjecture.
\end{thm}
\section{Preliminaries}

 Let $G$ be a finite $p$-group and $V$ a faithful
$\mathbb{F}_pG$-module. For a
subgroup $H\le G$ one sets a {\it j-function} of $H$
$$j_H(V):=\frac{|H||C_V(H)|}{|V|}\in \mathbb{Q}.$$
Note that $j_1(V) = 1$ and an elementary subgroup $E\ne 1$ is an offender if and only if $j_E(V)\ge 1$.
A subgroup $H\leq G$ is called
\emph{quadratic} on $V$ if $[V,H,H]=0$.
Since $V$ is faithful, every quadratic subgroup is
necessarily elementary abelian.
An offender that is quadratic on $V$ is called a
\emph{quadratic offender}.

Timmesfeld's replacement theorem is a fundamental tool in the theory of offenders and $F$-modules. It guarantees the existence of quadratic offenders whenever $V$ is an $F$-module. It therefore plays a central role in the study of the Quadratic Conjecture, where the main difficulty is to establish the existence of a quadratic element in the center of the group.
The following is a special case of Timmesfeld's
replacement theorem
\cite[Theorem~4.2]{ChermakDelgado1989}.

\begin{thm}
\label{thm:replacement}
Let $G$ be a finite $p$-group and $V$ a faithful
$\mathbb{F}_pG$-module. Suppose that $A$ is a
nontrivial abelian subgroup of $G$ satisfying
\[
    j_A(V)\geq j_D(V)
    \quad\text{for all }D\leq A.
\]
Set
\[
    B=C_A([V,A]).
\]
Then $B$ is a nontrivial elementary abelian
subgroup of $A$ satisfying
\[
    j_B(V)=j_A(V)
    \quad\text{and}\quad
    [V,B,B]=0.
\]
Moreover,
\[
    C_V(B)=[V,A]+C_V(A).
\]
\end{thm}
In particular,
by Timmesfeld's replacement theorem, a minimal offender must be a quadratic offender.

We shall make use of the following $j$-inequality
of Meierfrankenfeld and Stellmacher
\cite[Lemma~2.6]{MS2006}; see also
\cite[Lemma~4.4]{GHM2011}.
For any two subgroups $H,K\leq G$, we have
\begin{equation}\label{eq:j-inequality}
    j_{\langle H,K\rangle}(V)
    j_{H\cap K}(V)
    \geq j_H(V)j_K(V),
\end{equation}
with equality if and only if
\[
    \langle H,K\rangle=HK
    \quad\text{and}\quad
    C_V(H\cap K)=C_V(H)+C_V(K).
\]

Define
\[
    \mu(G,V)=\max\{j_H(V)\mid H\leq G\},
\]
and let
\[
    \operatorname{VCD}(G,V)
    =\{H\leq G\mid j_H(V)=\mu(G,V)\}.
\]
\begin{rem}\label{rem:mu}
If $V$ is an 
$F$-module, then it is clear that $\mu(G,V)\ge 1$.
\end{rem}

The following results describe the lattice structure
of $\operatorname{VCD}(G,V)$ and its relationship
with special offenders.

\begin{thm}\label{thm:VCD-offender}
Let $G$ be a finite $p$-group and $V$ a faithful
$\mathbb{F}_pG$-module.
Suppose that $H\in\operatorname{VCD}(G,V)$
and that $A$ is a nontrivial subgroup of $G$
satisfying
\[
    j_A(V)\geq j_B(V)
    \qquad\text{for every }B\leq A.
\]
Then
\[
    \langle H,A\rangle=HA
    \in\operatorname{VCD}(G,V)
\]
and
\[
    j_{H\cap A}(V)=j_A(V).
\]
Moreover,
\[
    C_V(H\cap A)=C_V(H)+C_V(A).
\]
\end{thm}

\begin{proof}
Set $\mu=\mu(G,V)$.
Since $H\in\operatorname{VCD}(G,V)$,
we have
\[
    j_H(V)=\mu.
\]
By the definition of $\mu$,
\[
    j_{\langle H,A\rangle}(V)\leq\mu.
\]
On the other hand, since $H\cap A\leq A$,
the hypothesis on $A$ gives
\[
    j_{H\cap A}(V)\leq j_A(V).
\]
Applying the $j$-inequality to $H$ and $A$,
we obtain
\[
\begin{aligned}
    \mu j_A(V)
    &=j_H(V)j_A(V)\\
    &\leq j_{\langle H,A\rangle}(V)
           j_{H\cap A}(V)\\
    &\leq \mu j_A(V).
\end{aligned}
\]
Therefore
\[
    j_{\langle H,A\rangle}(V)=\mu
\]
and
\[
    j_{H\cap A}(V)=j_A(V).
\]
In particular,
$\langle H,A\rangle\in\operatorname{VCD}(G,V)$.

Moreover, equality holds in the
$j$-inequality, which implies
\[
    \langle H,A\rangle=HA
\]
and
\[
    C_V(H\cap A)=C_V(H)+C_V(A).
\]
This completes the proof.
\end{proof}

\begin{thm}\label{thm:VCD-lattice}
Let $G$ be a finite $p$-group and $V$ a faithful
$\mathbb{F}_pG$-module. Then
$\operatorname{VCD}(G,V)$ is a modular sublattice
of the subgroup lattice of $G$.
More precisely, for any
$H,K\in\operatorname{VCD}(G,V)$, we have
\[
    H\cap K,\quad
    \langle H,K\rangle=HK
    \in\operatorname{VCD}(G,V),
\]
and
\[
    C_V(H\cap K)=C_V(H)+C_V(K).
\]
Moreover, $\operatorname{VCD}(G,V)$ has
a unique minimal element and a unique maximal
element, both of which are normal in $G$.
\end{thm}

\begin{proof} Since $K\in \operatorname{VCD}(G,V)$, $K$ satisfying
\[
    j_K(V)\geq j_B(V)
    \qquad\text{for every }B\leq K.
\]
By Theorem \ref{thm:VCD-offender},
$\langle H,K\rangle=HK
    \in\operatorname{VCD}(G,V)$
and
\[
    j_{H\cap K}(V)=j_K(V).
\]
Thus both $\langle H,K\rangle$ and $H\cap K$
belong to $\operatorname{VCD}(G,V)$. Moreover,
\[
    C_V(H\cap A)=C_V(H)+C_V(A).
\]

In particular, any two members of
$\operatorname{VCD}(G,V)$ permute.
The modularity of $\operatorname{VCD}(G,V)$
now follows from Dedekind's modular law.

Since $G$ is finite and
$\operatorname{VCD}(G,V)$ is closed under
intersections and joins, it has a unique
minimal element and a unique maximal element.

Finally, for every $g\in G$ and $H\leq G$,
we have
\[
    j_{H^g}(V)=j_H(V).
\]
Hence $\operatorname{VCD}(G,V)$ is invariant
under conjugation by $G$. Its unique minimal
and maximal elements are therefore normal
in $G$.
\end{proof}

\begin{cor}\label{cor:VCD-offender}
Let $G$ be a finite $p$-group, and $V$ an 
$F$-module.
Suppose that $M$ is the maximal element in $\operatorname{VCD}(G,V)$, and $A$ is an inclusion minimal offender. Then $A^G\le M$.
\end{cor}
\begin{proof} Since $A$ is inclusion minimal, $j_A(V)\geq j_B(V)$
for every $B\leq A$. By Theorem \ref{thm:VCD-offender}, $MA\in \operatorname{VCD}(G,V)$. By the maximality of $M$, $A\le M$. Hence $A^G\le M$.
\end{proof}
Let $\Lambda$ be a lattice. For $X,Y\in\Lambda$,
we write $X\prec Y$ if $X<Y$ and there is no
$Z\in\Lambda$ such that $X<Z<Y$.
In this case, we say that $Y$ covers $X$ in $\Lambda$.

\begin{thm}\label{thm:VCD-cover}
Let $G$ be a finite $p$-group and $V$ a faithful
$\mathbb{F}_pG$-module. If $H$ covers $K$ in
$\operatorname{VCD}(G,V)$, then $K\unlhd H$.
\end{thm}

\begin{proof}
Suppose, for a contradiction, that $K$ is not normal
in $H$. Set $N=N_H(K)<H$. Since $H$ is a finite
$p$-group, the normalizer condition yields
\[
    N<N_H(N).
\]
Choose $x\in N_H(N)\setminus N$. Then
$K^x\leq N$, but $K^x\neq K$.

Since $\operatorname{VCD}(G,V)$ is invariant under
conjugation and closed under joins, we have
\[
    \langle K,K^x\rangle\in\operatorname{VCD}(G,V).
\]
Consequently,
\[
    K<\langle K,K^x\rangle\leq N<H,
\]
contrary to the assumption that $H$ covers $K$
in $\operatorname{VCD}(G,V)$.
Therefore $K\unlhd H$.
\end{proof}

To prove our main theorem, we need the following lemma.
%

\begin{lem}\label{lem:class}{\cite[Corollary~9.21]{Lam2001}}
Let $G$ be a finite $p$-group and $V$ a faithful
$\mathbb{F}_pG$-module of dimension $m\geq 1$.
Then
\[
    c(G)\leq m-1,
\]
where $c(G)$ denotes the nilpotency class of $G$.
\end{lem}

\section{The proof of Theorem \ref{thm:maximal_class}}

Let \(G\) be a \(p\)-group of maximal class, \(V\) an \(F\)-module. Assume \(G\) does not satisfy the quadratic conjecture. Let \(A\) be an inclusion minimal offender in \(G\) and $N = \langle A^G \rangle$. Let
\begin{center}$U=V/C_V(N)$ and $K=C_G(U)$.
\end{center}
If \(K \neq 1\), then $Z(G)\le K$ since $G$  is of maximal class. The same reason gives that $Z(G)\le N$. Hence, for \(1 \neq z\in Z(G)\), $[V,z,z]\leq [C_V(N),z]=1.$ This contradicts the fact that there is no quadratic element in $\Omega_1(Z(G))$. 
Hence \(K=1\). That is, $U$ is a faithful
$\mathbb{F}_pG$-module. By Lemma \ref{lem:class}, \(C(G) \le \dim U-1\). Thus $|G|\le p^{\dim U}$. By calculation,
\begin{equation}\label{eq:maximal}
j_N(V)=\frac{|N||C_V(N)|}{|V|}=\frac{|N|}{p^{\dim U}}\le\frac{|N|}{|G|}\le 1.
\end{equation}
Let \(M\) be the maximal element in \(\operatorname{VCD}(N,V)\). By Corollary \ref{cor:VCD-offender}, $N\le M$. That is, \(N\) is the maximal element in \(\operatorname{VCD}(N,V)\). By Remark \ref{rem:mu}, $j_N(V)\ge 1$. By inequality \ref{eq:maximal},
\begin{center} $G=N$ and $j_N(V)=1$.
\end{center}
By the definition of \(\operatorname{VCD}(G,V)\), $A\in \operatorname{VCD}(G,V)$.

If $A\trianglelefteq G$, then $Z(G)\le A$. Since $A$ is inclusion minimal, by Theorem \ref{thm:replacement}, $A$ is quadratic. Hence $Z(G)$ is also quadratic, which contradicts that there is no quadratic element in $\Omega_1(Z(G))$. Hence $A$ is not normal in $G$.

For all \(g \in G\) such that \(A^g \neq A\), we have $A, A^g\in \operatorname{VCD}(G,V)$. Hence both $AA^g$ and $A\cap A^g$ are in $A, A^g\in \operatorname{VCD}(G,V)$. Since \(A\) is inclusion minimal, \[A \cap A^g = 1.\]
The minimality of $A$ also imply that both $A$ and $A^g$ cover $1$. Since \(\operatorname{VCD}(G,V)\) is modular, $AA^g$ covers both $A$ and $A^g$.
By Theorem \ref{thm:VCD-cover}, both $A$ and $A^g$ are normal in $AA^g$. Hence \([A, A^g] = A \cap A^g = 1\). 
By the arbitrariness of $g$, $G=N=A^G$ is abelian, which contradicts that $G$ is of maximal class.

%
%
%

\section*{Acknowledgments}

The authors acknowledge the assistance of ChatGPT
(OpenAI) in improving the presentation of this paper,
including English-language editing, LaTeX preparation,
and discussions of mathematical arguments and references.
The authors take full responsibility for all mathematical
results, proofs, and citations.

\end{document}